\documentclass[reqno,12pt]{amsart}
\usepackage{amssymb,amsmath,amsthm,amsfonts,amssymb,latexsym}
\usepackage{hyperref,enumerate,mathtools,xparse}
\newtheorem{coro}{Corollary}[section]
\newtheorem{theorem}{Theorem}[section]
\newtheorem{lemma}[theorem]{Lemma}
\theoremstyle{remark}

\numberwithin{equation}{section}

\usepackage[left=1in, right=1in, bottom= 1in, top=1in]{geometry}
\begin{document}
\title [The eighth order mock theta function $V_0(q)$]{Congruence properties of coefficients of the eighth order mock theta function $V_0(q)$}
\author{B. Hemanthkumar}
\address[B. Hemanthkumar]{Department of Mathematics, M. S. Ramaiah University of Applied Sciences, Peenya Campus, Peenya 4th Phase, Bengaluru-560 058, Karnataka, India.} \email{hemanthkumarb.30@gmail.com}

\begin{abstract}
We study the divisibility properties of the partition function associated with the eighth order mock theta function $V_0(q)$, introduced by Gordon and McIntosh. 
We obtain congruences modulo powers of 2 for certain coefficients of the partition function, akin to Ramanujan's partition congruences. Further, we also present several infinite families of congruences molulo 13, 25 and 27.
\end{abstract}
	
\maketitle
\section{Introduction}	
Ramanujan in his last letter to G. H. Hardy recorded 17 mock theta functions, which he classified into three classes of orders 3, 5 and 7 (The order is analogous to the level of a modular form). Later, Watson \cite{W} found 5 more functions of orders 3 and 5, which appear in the lost notebook \cite{SR}. Thenceforth, other mock theta functions of different orders have been found.

In \cite{GM}, Gordon and McIntosh established eight new mock theta functions of order 8, one of which is $V_0(q)$, defined by
\begin{equation*}
V_0(q)=-1+2\sum_{n=0}^\infty \frac{q^{{n^2}} (-q;q^2)_n}{(q;q^2)_n}=-1+2\sum_{n=0}^\infty \frac{q^{2n^2} (-q^2;q^4)_n}{(q;q^2)_{2n+1}}
\end{equation*}
and
\begin{equation}\label{S1}
V_0(q)+V_0(-q)=2(-q^2; q^4)^4_\infty (q^8; q^8)_\infty.
\end{equation}
Here we use the standard notation:
\begin{equation*}
(a; q)_0 = 1, \quad (a; q)_n = \prod_{i=1}^n (1-aq^{i-1})  \quad \text{and} \quad (a; q)_\infty = \lim_{n\rightarrow \infty}(a; q)_n  \quad |q|<1.
\end{equation*}
The function $V_0(q)$ is rather interesting. It has an interpretation as the generating function for $v_0(n)$, number of overpartitions of $n$ into odd parts without gaps between the non--overlined parts. It is given by
\begin{equation}\label{S11} 
\sum_{n=0}^\infty v_0(n) q^n = \sum_{n=0}^\infty \frac{q^{{n^2}} (-q;q^2)_n}{(q;q^2)_n} = \frac{V_0(q)+1}{2}.
\end{equation}
It is easy from \eqref{S1} and \eqref{S11} that
\begin{equation*}
\frac{1}{2}+\sum_{n=1}^\infty v_0(2n) q^n =\frac{(q^2;q^2)_\infty}{2(q;q)_\infty (q^4; q^4)^3_\infty}.
\end{equation*}
Further, other combinatorial interpretatitons to $V_0(q)$ were given by Agarwal and Sood \cite{AS} using split $(n+t)$--color partitions, and by Rana and Sareen \cite{RS} in terms of signed partitions.

Congruence properties of partition functions related with mock theta functions have been the interest of recent study. For instance, Andrews, Dixit and Yee \cite{ADY} introduced the partition functions $p_\omega(n)$ and $p_\nu(n)$ associated with the third order mock theta functions $\omega(q)$ and $\nu(q)$ defined by
\begin{equation*}
\omega(q) = \sum_{n=0}^\infty \frac{q^{2n^2+2n}}{(q;q^2)^2_{n+1}} \quad \text{and} \quad \nu(q) = \sum_{n=0}^\infty \frac{q^{n^2+n}}{(-q;q^2)_{n+1}}.
\end{equation*}
While, Andrews \textit{et al}. \cite{APSY} obtained infinite families of congruences modulo 4 and 8 for the functions $p_\omega(n)$ and $p_\nu(n)$.

Recently, applying the theory of (mock) modular forms and Zwegers' results on Appell-Lerch sums, Mao \cite{M} established two identities for $V_0(q)$, in terms of $v_0(n)$ is given by
\begin{equation*}
\sum_{n=0}^\infty v_0(8n+2) = 2 \frac{f_2^4 f_4^5}{f_1^6 f_8^2}  \quad \text{and} \quad \sum_{n=0}^\infty v_0(8n+6) = 4 \frac{f_2^2 f_4^4}{f_1^5}.
\end{equation*} 
As a consequence, showed that for all $n\geq 0$
\begin{equation*}
 v_0(40n+13) \equiv   v_0(40n+37) \equiv 0 \pmod{40}.
\end{equation*} 
At the same time, Brietzke, Silva and Sellers \cite{BSS} using elementary generating function manipulations, presented the parity of $v_0(n)$, identities involving the generating function for $v_0(n)$ and many congruences modulo certain number of the form $2^\alpha 3^\beta 5$. For example, they obtain 
\begin{equation}\label{E119}
\frac{1}{2}+\sum_{n=1}^\infty v_0(4n)q^n= \frac{\varphi(q)^2} {2\varphi(-q)},
\end{equation}
as a consequence showed that for all $n\geq 0$ and prime $p$,
\begin{equation}
v_0(4n) \equiv \begin{cases}
(-1)^k \pmod{4},\quad \text{if}\,\,\, n=k^2, \\
   \qquad  0 \pmod{4}, \quad \text{otherwise}		\label{E0} 
\end{cases}
\end{equation} 
and
\begin{equation}
v_0(4(pn+r)) \equiv 0 \pmod{4} \label{E00}
\end{equation}
if $r$ is a quadratic nonresidue modulo $p$. \\ They also prove that 
\begin{equation}\label{S13}
v_0(16n+12) \equiv 0 \pmod{2^4} \quad \text{and} \quad v_0(32n+28) \equiv 0 \pmod{2^6}.\end{equation}
In this paper, we probe arithmetic properties of the partition function $v_0(n)$, 
and prove several infinite families of congruences modulo powers of 2. In fact, we generalize the results \eqref{E119}-\eqref{S13} and also obtain many Ramanujan type congruences modulo 13, 25 and 27.

The main results of the paper are as follow:

\begin{theorem}\label{T2}
 For any integer $n\geq0$ and $\alpha>0$, we have
\begin{equation}\label{E1} 
v_0(2^{2\alpha+2}n)\equiv \begin{cases}
1 \pmod{2^{3\alpha+2}},\quad {\rm{if}}\,\,\, n=k^2, \\
0 \pmod{2^{3\alpha+2}}, \quad {\rm{otherwise}}.		
\end{cases}
\end{equation}
\end{theorem}
The following corollary is an immediate consequence of the above theorem.
\begin{coro}\label{C1}
For any integer $\alpha>0$ and prime $p\geq3$, if $r$ is a quadratic nonresidue modulo $p$, then 
\begin{equation}\label{E2}
v_0(2^{2\alpha+2}(pn+r)) \equiv 0 \pmod {2^{3\alpha+2}}.
\end{equation}
\end{coro}
\begin{theorem}\label{T3}
 For all integers $n\geq0$ and $\alpha\geq0$, we have
\begin{align}
v_o(2^{2\alpha+1}(2n+1))  &\equiv 0 \pmod {2^{3\alpha}}, \label{E3} \\
v_o(2^{2\alpha+2}(4n+3))  &\equiv 0 \pmod {2^{3\alpha+4}}, \label{E4} \\
v_o(2^{2\alpha+2}(8n+7))  &\equiv 0 \pmod {2^{3\alpha+6}}, \label{E5} \\
v_o(2^{2\alpha+2}(8n+5))  &\equiv 0 \pmod {2^{3\alpha+6}}. \label{E55} 
\end{align}
\end{theorem}
\begin{theorem}\label{T4}
If $n$ cannot be represented as a triangular number, then for any integer $\alpha\geq0$, $\beta\geq0$ and  prime $p\geq3$,  
\begin{equation}
v_o(2^{2\alpha+2}p^{2\beta}(8n+1))  \equiv 0 \pmod {2^{3\alpha+9}}. \label{E6} 
\end{equation}
\end{theorem}
\begin{coro}\label{C2}
For any integer $n\geq0$, $\alpha\geq0$, $\beta \geq0$ and odd prime $p$, we have
\begin{equation}
v_o(2^{2\alpha+2}p^{2\beta+1}(8pn+8i+p))  \equiv 0 \pmod {2^{3\alpha+9}}, \,\ i\in\{1,2,3,\ldots,p-1\} \label{E8} 
\end{equation} 
and 
\begin{equation}
v_o(2^{2\alpha+2}p^{2\beta}(8pn+8j+1))  \equiv 0 \pmod {2^{3\alpha+9}}, \label{E9} 
\end{equation} 
where $j$ is an integer with $0\leq j \leq p-1$ such that $\left(\frac{8j+1}{p}\right)=-1$.
\end{coro}
\section{Preliminaries}\label{S2}
\noindent In this section, we provide some preliminary results which play an important role in proving our results. We start by considering Ramanujan's theta function $f(a,b)$, defined by
\begin{equation*}
f(a,b)=\sum_{n=-\infty}^\infty a^{n(n+1)/2} b^{n(n-1)/2},  \,\,\ |ab|<1.
\end{equation*} 
Jacobi's triple product identity takes the form 
\begin{equation*}
f(a,b) = (-a;ab)_\infty (-b;ab)_\infty (ab;ab)_\infty.
\end{equation*}
The three special cases of $f(a,b)$ are
\begin{align*}
\psi(q)&:=f(q,q^3)=\sum_{n=0}^\infty q^{n(n+1)/2}=\frac{(q^2;q^2)_\infty}{(q;q^2)_\infty},\\
\varphi(q)&:=f(q,q)=\sum_{n=-\infty}^\infty q^{n^2} = (-q;q^2)^2_\infty (q^2;q^2)_\infty
\end{align*}
and
\begin{equation*}
f(-q):=f(-q,-q^2)=\sum_{n=-\infty}^\infty (-1)^n q^{n(3n-1)/2} = (q;q)_\infty.
\end{equation*}
For brevity we shall also use $f_n$ to denote $f(-q^n)$.

Consider the identities \cite[Entry 25, p. 40]{BCB3} due to Ramanujan, 
\begin{align}
\varphi(-q^2)^2&=\varphi(q)\varphi(-q), \label{L11}\\
\varphi(-q)&=\varphi(q^4)-2q\psi(q^8), \label{L12}\\
\varphi(-q)^2&=\varphi(q^2)^2-4q\psi(q^4)^2. \label{L13}
\end{align}
In \eqref{L12} we replace $q$ by $q$ and $-q$, and multiply the two results, we obtain
\begin{equation}
\varphi(q^4)^2-4q^2\psi(q^8)^2 = \varphi(-q^2)^2. \label{E11}
\end{equation}
We now let
\begin{equation}\label{E12}
\zeta= \frac{\varphi(-q)}{\varphi(q^4)},\quad T = \frac{\varphi(-q^2)^2}{\varphi(q^4)^2},\quad \xi=\frac{\varphi(-q)^2}{\varphi(q^2)^2}, \quad G= \frac{\varphi(-q^2)^4}{\varphi(q^2)^4}.
\end{equation}
Then \eqref{L12} gives
\begin{equation}\label{E13}
\zeta = 1-2q\frac{\psi(q^8)}{\varphi(q^4)}
\end{equation}
and \eqref{E11} gives
\begin{equation}\label{E14}
T = 1-4q^2\frac{\psi(q^8)^2}{\varphi(q^4)^2}.
\end{equation}
By \eqref{E13},
\begin{equation}\label{E15}
\zeta^2 = 1-4q\frac{\psi(q^8)}{\varphi(q^4)}+4q^2\frac{\psi(q^8)^2}{\varphi(q^4)^2}.
\end{equation}
From \eqref{E15}, \eqref{E14} and \eqref{E13}, it follows that
\begin{equation*}
\zeta^2-2\zeta+T=0.
\end{equation*}
The last identity can be rewritten as
\begin{equation*}
\frac{1}{\zeta}=\frac{1}{T}\left(2-\zeta\right),
\end{equation*}
thus for any $i\geq 1$,
\begin{equation}\label{E16}
\frac{1}{\zeta^i}=\frac{1}{T}\left(\frac{2}{\zeta^{i-1}}-\frac{1}{\zeta^{i-2}}\right).
\end{equation}
Using \eqref{L13}, similar way as above, one can easily see that
\begin{equation}\label{E17}
\xi = 1-4q\frac{\psi(q^4)^2}{\varphi(q^2)^2},
\end{equation}
and for all $i\geq 1$,
\begin{equation}\label{E18}
\frac{1}{\xi^i}=\frac{1}{G}\left(\frac{2}{\xi^{i-1}}-\frac{1}{\xi^{i-2}}\right).
\end{equation}
Let $H$ be the Huffing operator, given by
\begin{equation*}
H\left(\sum_{n=0}^\infty a_n q^n\right)=\sum_{n=0}^{\infty}a_{2n}q^{2n}.
\end{equation*}
From \eqref{E13}, we have
\begin{equation*}
H(\zeta)=1
\end{equation*}
and, clearly, $H(1)=1.$

Also, in view of \eqref{E16}, we see that
\begin{equation}\label{E19}
H\left(\frac{1}{\zeta^i}\right)=\frac{1}{T}\left(H\left(\frac{2}{\zeta^{i-1}}\right)-H\left(\frac{1}{\zeta^{i-2}}\right)\right).
\end{equation}
In particular,
\begin{align}
H\left(\frac{1}{\zeta}\right)&=\frac{1}{T}(2-1)=\frac{1}{T},\label{E1112}\\
H\left(\frac{1}{\zeta^2}\right)&=\frac{1}{T}\left(2\times\frac{1}{T}-1\right)=-\frac{1}{T}+\frac{2}{T^2}\label{E1113}
\end{align}
and
\begin{equation}
H\left(\frac{1}{\zeta^3}\right)=\frac{1}{T}\left(2\left(\frac{-1}{T}+\frac{2}{T^2}\right)-\frac{1}{T}\right)=-\frac{3}{T^2}+\frac{4}{T^3}\label{E111}.
\end{equation}
We note from \eqref{E111} that
\begin{equation*}
H\left(\frac{\varphi(q^4)^3}{\varphi(-q)^3}\right)=-3\frac{\varphi(q^4)^4}{\varphi(-q^2)^4}+4\frac{\varphi(q^4)^6}{\varphi(-q^2)^6},
\end{equation*}
which yeilds,
\begin{equation}\label{E118}
H\left(\frac{\varphi(q)^2}{\varphi(-q)}\right)=-3\varphi(q^4)+4\frac{\varphi(q^4)^3}{\varphi(-q^2)^2},
\end{equation}
Thus, from \eqref{E119} and \eqref{E118}, 
\begin{equation}\label{E1114}
\frac{1}{2}+\sum_{n=1}^\infty v_0(8n)q^n = \frac{-3}{2}\varphi(q^2)+2\frac{\varphi(q^2)^3}{\varphi(-q)^2}.
\end{equation}
Using \eqref{L13} in the last identity and extracting the terms involving even powers of $q$, we obtain
\begin{equation}\label{E1115}
\frac{1}{2}+\sum_{n=1}^\infty v_0(16n)q^n = \frac{-3}{2}\varphi(q)+2\frac{\varphi(q)^5}{\varphi(-q)^4}.
\end{equation}
It is evident from \eqref{E1112}-\eqref{E111} and \eqref{E19} that for $i\geq1$, 
\begin{equation}\label{E112}
H\left(\frac{1}{\zeta^i}\right)=\sum_{j=1}^{i} \frac{m_{i,j}}{T^j}, 
\end{equation}
where the $m_{i,j}$ form a matrix $M$ given as follows:
\begin{enumerate}
\item \label{MR1} $m_{1,1}=1,\ m_{1,j}=0$ for all $j \geq 2$.
\item \label{MR2} $m_{2,1}=-1,\ m_{2,2}=2,\ m_{2,j}=0,$ for all $j \geq 3$.
\item \label{MR3} $m_{i,1}=0$ for all $i \geq 3$.
\item \label{MR4} $m_{i,j}=2m_{i-1,j-1}-m_{i-2,j-1}$ for all $i\geq3\ j\geq 2$.
\end{enumerate}
The first eight rows of $M$ are given by
\begin{equation*}
\left (\begin{matrix} 
1&0&0&0&0&0&0&0&0&\cdots\\
-1&2&0&0&0&0&0&0&0&\cdots\\
0&-3&4&0&0&0&0&0&0&\cdots\\
0&1&-8&8&0&0&0&0&0&\cdots\\
0&0&5&-20&16&0&0&0&0&\cdots\\
0&0&-1&18&-48&32&0&0&0&\cdots\\
0&0&0&-7&56&-112&64&0&0&\cdots\\
0&0&0&1&-32&160&-256&128&0&\cdots\\
\end{matrix}\right ).\\
\end{equation*}
Similarly, using \eqref{E17} and \eqref{E18}, we can see that for $i\geq1$, 
\begin{equation}\label{E113}
H\left(\frac{1}{\xi^i}\right)=\sum_{j=1}^{i} \frac{m_{i,j}}{G^j}.
\end{equation}
Actually, $m_{2i-1,j}=0$ for $1\leq j<i$, so we can write
\begin{equation}\label{E114}
H\left(\frac{1}{\zeta^{2i-1}}\right)=\sum_{j=1}^{i} \frac{a_{i,j}}{T^{j+i-1}},
\end{equation}
where 
\begin{equation}\label{E115}
a_{i,j} = m_{2i-1, j+i-1}.
\end{equation}
We can rewrite \eqref{E114} as
\begin{equation*}
H\left(\left(\frac{\varphi(q^4)}{\varphi(-q)}\right)^{2i-1}\right)=\sum_{j=1}^{i}a_{i,j}\left(\frac{\varphi(q^4)^2}{\varphi(-q^2)^2}\right)^{j+i-1}
\end{equation*}
or
\begin{equation}\label{E116}
H\left(\frac{\varphi(q)^i}{\varphi(-q)^{i-1}}\right)=\sum_{j=1}^{i}a_{i,j}\frac{\varphi(q^4)^{2j-1}}{\varphi(-q^2)^{2j-2}},
\end{equation}
and \eqref{E113} as
\begin{equation}\label{E117}
H\left(\frac{\varphi(q^2)^{2i+1}}{\varphi(-q)^{2i}}\right)=\sum_{j=1}^{i}m_{i,j}\frac{\varphi(q^2)^{4j+1}}{\varphi(-q^2)^{4j}}.
\end{equation}
Now let
\begin{equation}\label{E122}
\mu= \frac{\varphi(-q)}{q\psi(q^8)},\quad S = \frac{\varphi(-q^2)^2}{q^2\psi(q^8)^2},\quad \rho=\frac{\varphi(-q)^2}{q\psi(q^4)^2}, \quad F= \frac{\varphi(-q^2)^4}{q^2\psi(q^4)^4}.
\end{equation}
Then \eqref{L12} gives
\begin{equation*}
\mu = \frac{\varphi(q^4)}{q\psi(q^8)}-2
\end{equation*}
and \eqref{E11} gives
\begin{equation*}
S = \frac{\varphi(q^4)^2}{q^2\psi(q^8)^2}-4.
\end{equation*}
It follows that
\begin{equation*}
\mu^2+4\mu-S=0
\end{equation*}
or
\begin{equation*}
\frac{1}{\mu}=\frac{1}{S}(\mu+4).
\end{equation*}
Thus for $i\geq 1$,
\begin{equation*}
\frac{1}{\mu^i}=\frac{1}{S}\left(\frac{1}{\mu^{i-2}}+\frac{4}{\mu^{i-1}}\right)
\end{equation*}
and 
\begin{equation*}
H\left(\frac{1}{\mu^i}\right)=\sum_{j\geq1} \frac{n_{i,j}}{S^j},
\end{equation*}
where the $n_{i,j}$ form a matrix $N$ defined as follows,
\begin{enumerate}
	\setcounter{enumi}{4}
	\item \label{MR5} $n_{1,1}=2,\ n_{1,j}=0$ for all $j \geq 2$.
	\item \label{MR6} $n_{2,1}=1,\ n_{2,2}=8,\ n_{2,j}=0,$ for all $j \geq 3$.
	\item \label{MR7} $n_{i,1}=0$ for all $i \geq 3$.
	\item \label{MR8} $n_{i,j}=4n_{i-1,j-1}+n_{i-2,j-1}$ for all $i\geq3\ j\geq 2$.
\end{enumerate}
In fact,  $n_{2i-1,j}=0$ for $1\leq j<i$, so we can write
\begin{equation}\label{E214}
H\left(\frac{1}{\mu^{2i-1}}\right)=\sum_{j=1}^{i} \frac{c_{i,j}}{S^{j+i-1}},
\end{equation}
where 
\begin{equation}\label{E215}
c_{i,j} = n_{2i-1, j+i-1}.
\end{equation}
We can rewrite \eqref{E214} as
\begin{equation*}
H\left(\left(q\frac{\psi(q^8)}{\varphi(-q)}\right)^{2i-1}\right)=\sum_{j=1}^{i}c_{i,j}\left(q^2\frac{\psi(q^8)^2}{\varphi(-q^2)^2}\right)^{j+i-1}
\end{equation*}
or
\begin{equation}\label{E216}
H\left(q\frac{\varphi(q)^i}{\varphi(-q)^{i-1}}\right)=\sum_{j=1}^{i}c_{i,j} \ q^{2j}\frac{\psi(q^8)^{2j-1}}{\varphi(-q^2)^{2j-2}}.
\end{equation}
In particular, 
\begin{equation}\label{E217}
H\left(q\frac{\varphi(q)^2}{\varphi(-q)}\right)=6 q^2 \psi(q^8) +32 q^4 \frac{\psi(q^8)^3}{\varphi(-q^2)^2}.
\end{equation}
Multiplying $q$ on both sides of \eqref{E119}, 
\begin{equation*}
\frac{q}{2}+\sum_{n=1}^\infty v_0(4n)q^{n+1}= q\frac{\varphi(q)^2} {2\varphi(-q)}.
\end{equation*}
 From \eqref{E217} and above identity,
\begin{equation} \label{T0}
\sum_{n=0}^\infty v_0(8n+4)q^n = 3\psi(q^4)+16q\frac{\psi(q^4)^3}{\varphi(-q)^2}.
\end{equation} 

Utilizing \eqref{L13}, similar way as above, we readily see that for $i\geq1,$
\begin{equation*}
\frac{1}{\rho^i}=\frac{1}{F}\left(\frac{1}{\rho^{i-2}}+\frac{8}{\rho^{i-1}}\right).
\end{equation*}
and 
\begin{equation*}
H\left(\frac{1}{\rho^i}\right)=\sum_{j\geq1} \frac{p_{i,j}}{F^j}
\end{equation*}
or
\begin{equation*}
H\left(q^i\frac{\psi(q^4)^{2i+1}}{\varphi(-q)^{2i}}\right)=\sum_{j=1}^{i}p_{i,j} \ q^{2j}\frac{\psi(q^4)^{4j+1}}{\varphi(-q^2)^{4j}}.
\end{equation*}
When $i=1$, we get
\begin{equation*}
H\left(q\frac{\psi(q^4)^{3}}{\varphi(-q)^{2}}\right)= 4 q^2 \frac{\psi(q^4)^5}{\varphi(-q^2)^4}.
\end{equation*}
By \eqref{T0} and the last identity,
\begin{equation}\label{T00}
\sum_{n=0}^\infty v_0(16n+4)q^n = 3\psi(q^2)+64q\frac{\psi(q^2)^5}{\varphi(-q)^4}.
\end{equation}
where $p_{i,j}$ form a matrix $P$ defined as follows,
\begin{enumerate}
	\setcounter{enumi}{8}
	\item \label{MR9} $p_{1,1}=4,\ p_{1,j}=0$ for all $j \geq 2$.
	\item \label{MR10} $p_{2,1}=1,\ p_{2,2}=32,\ p_{2,j}=0,$ for all $j \geq 3$.
	\item \label{MR11} $p_{i,1}=0$ for all $i \geq 3$.
	\item \label{MR12} $p_{i,j}=8p_{i-1,j-1}+p_{i-2,j-1}$ for all $i\geq3\ j\geq 2$.
\end{enumerate}
In \cite{CG}, Cui and Gu studied $p$ - dissection identities for $\psi(q)$..\\
For any odd prime $p$,
\begin{equation}\label{EE1}
\psi(q) = \sum_{k=0}^{\frac{p-3}{2}} q^{\frac{k^2+k}{2}} f(q^{\frac{p^2+(2k+1)p}{2}}, q^{\frac{p^2-(2k+1)p}{2}}) + q^{\frac{p^2-1}{8}}\psi(q^{p^2}).
\end{equation}
Futhermore, if $0<k<\frac{p-3}{2}$,
\begin{equation*}
\frac{k^2+k}{2}\not\equiv \frac{p^2-1}{8} \pmod{p}. 
\end{equation*}
Let us suppose
\begin{equation}
\sum_{n=0}^\infty a(n) q^n = \psi(q). \label{EE11}
\end{equation}
Then, it is evident from \eqref{EE1} and mathematical induction that, for any odd prime and $\alpha \geq 0$,
\begin{align}
\sum_{n=0}^\infty a\left(p^{2\alpha} +\frac{p^{2\alpha}-1}{8}\right)q^n=\psi(q),\label{EE2}\\ 
 a\left(p^{2\alpha+2}n +\frac{(8i+p)p^{2\alpha+1}-1}{8}\right)=0\label{EE3}
\end{align}
and
\begin{equation}
 a\left(p^{2\alpha+1}n +\frac{(8j+1)p^{2\alpha}-1}{8}\right)=0,\label{EE4}
\end{equation}
where \quad $i \in \{1, 2, \ldots , p-1\}$ \, and \, $j$ is an integer with $0\leq j \leq p-1$ such that $\left(\frac{8j+1}{p}\right)=-1$.

\section{Generating functions}\label{S3}
In this section, we obtain generating functions for the sequences in Theorems \ref{T2} - \ref{T4}.
Let $A, B, C$ and $D$ be matrices defined as below,
\begin{align*}
A&=(a_{4i-3, j})_{i,j\geq1},\\
B&=(b_{i, j})_{i,j\geq1}
\end{align*}
where $b_{1,1}=1$, $b_{k,1}=b_{1,k}=0$, for $k>1$ and $b_{i,j}=m_{i-1,j-1}$, for $i, j\geq2$,
\begin{equation*}
C = (c_{4i-3, j})_{i,j\geq1}
\end{equation*}
and
\begin{equation*}
D=(d_{i, j})_{i,j\geq1}
\end{equation*}
where $d_{1,1}=1$, $d_{k,1}=d_{1,k}=0$, for $k>1$ and $d_{i,j}=p_{i-1,j-1}$, for $i, j\geq2$.

We now define the coefficient vectors $\textbf{x}_\alpha, \ \textbf{y}_\alpha, \ \textbf{z}_\alpha, \,\ \alpha\geq1$ by 
\begin{equation}\label{E7} 
\textbf{x}_1=(-3/2, \ 2,\ 0,\ 0,\ 0,\ \ldots)
\end{equation}
and for $\alpha\geq1,$
\begin{align}
\textbf{x}_{2\alpha}&=\textbf{x}_{2\alpha-1}A,\label{311}\\
\textbf{x}_{2\alpha+1}&=\textbf{x}_{2\alpha}B,\label{312}\\
\textbf{y}_\alpha &= \textbf{x}_{2\alpha-1}C,\label{313}\\
\textbf{z}_\alpha &= \textbf{y}_{\alpha}D.\label{314}
\end{align}
In particular, the vectors $\textbf{x}_2$ and $\textbf{x}_3$ are given by
\begin{align}
\textbf{x}_2&=(33/2, \,\ -2^4\times3\times5,\,\ 2^5\times3^3,\,\ -2^7\times3^2,\,\ 2^9,\, 0, \,\ 0, \ \ldots). \label{E123}\\
\textbf{x}_3&=(33/2, \,\ -2^4\times3\times23,\,\ 2^6\times89,\,\ -2^9\times17,\,\ 2^{12},\,\ 0, \,\ 0 \ \ldots).\label{E124}
\end{align}

We are now in a stage to prove the following theorem:
\begin{theorem} \label{T1}
For any integer $\alpha > 0,$
\begin{align}
\frac{1}{2}+\sum_{n=1}^\infty v_0(2^{2\alpha+2}n)q^n &= \sum_{j\geq1} x_{2\alpha-1,  j}  \frac{\varphi(q)^{4j-3}}{\varphi(-q)^{4j-4}},\label{T11}\\
\frac{1}{2}+\sum_{n=1}^\infty v_0(2^{2\alpha+3}n)q^n &= \sum_{j\geq1} x_{2\alpha,  j} \frac{\varphi(q^2)^{2j-1}}{\varphi(-q)^{2j-2}},\label{T12}\\
\sum_{n=0}^\infty v_0(2^{2\alpha+2}(2n+1))q^n &=\sum_{j\geq1} y_{\alpha,  j} \ q^{j-1} \frac{\psi(q^4)^{2j-1}}{\varphi(-q)^{2j-2}}\label{T13}
\end{align}
and 
\begin{equation}
\sum_{n=0}^\infty v_0(2^{2\alpha+2}(4n+1))q^n =\sum_{j\geq1} z_{\alpha,  j} \ q^{j-1} \frac{\psi(q^2)^{4j-3}}{\varphi(-q)^{4j-4}}.\label{T14}
\end{equation}
\end{theorem}
\begin{proof} 
By \eqref{E1115}, we note that identity \eqref{T11} holds for $\alpha=1$. 
Suppose that \eqref{T11} holds for some $\alpha\geq1$. Applying operator $H$ to both sides gives
\begin{align*}
\frac{1}{2}+\sum_{n=1}^\infty v_0(2^{2\alpha+3}n)q^{2n} &= \sum_{j\geq1} x_{2\alpha-1,  j} H\left(\frac{\varphi(q)^{4j-3}}{\varphi(-q)^{4j-4}}\right)\\
& = \sum_{j\geq1} x_{2\alpha-1, j} \sum_{k\geq1} a_{4j-3,k}\frac{\varphi(q^4)^{2k-1}}{\varphi(-q^2)^{2k-2}}\\
&= \sum_{k\geq1} \left(\sum_{j\geq1} x_{2\alpha-1, j} \ a_{4j-3, k}\right)\frac{\varphi(q^4)^{2k-1}}{\varphi(-q^2)^{2k-2}}\\
&= \sum_{k\geq1} x_{2\alpha, k}\frac{\varphi(q^4)^{2k-1}}{\varphi(-q^2)^{2k-2}}.
\end{align*}
If we now change $q^2$ by $q$, we obtain
\begin{equation*}
\frac{1}{2}+\sum_{n=1}^\infty v_0(2^{2\alpha+3}n)q^{n}= \sum_{k\geq1} x_{2\alpha, k}\frac{\varphi(q^2)^{2k-1}}{\varphi(-q)^{2k-2}},
\end{equation*}
which is in fact \eqref{T12}.\\
Now suppose \eqref{T12} holds for some $\alpha\geq1$. If we apply the operator $H$ to both sides gives
\begin{align*}
\frac{1}{2}+\sum_{n=1}^\infty v_0(2^{2\alpha+4}n)q^{2n} &= \sum_{j\geq1} x_{2\alpha,  j} H\left(\frac{\varphi(q^2)^{2j-1}}{\varphi(-q)^{2j-2}}\right)\\
&=x_{2\alpha, 1} \ \varphi(q^2)+\sum_{j\geq1} x_{2\alpha,  j+1} H\left(\frac{\varphi(q^2)^{2j+1}}{\varphi(-q)^{2j}}\right)\\
&=x_{2\alpha, 1} \ \varphi(q^2)+\sum_{j\geq1} x_{2\alpha,  j+1} \sum_{k\geq1}m_{j,k}\frac{\varphi(q^2)^{4k+1}}{\varphi(-q^2)^{4k}}\\
&=x_{2\alpha, 1} \ \varphi(q^2)+\sum_{k\geq1}  \left(\sum_{j\geq1} x_{2\alpha,  j+1} \ m_{j,k}\right)\frac{\varphi(q^2)^{4k+1}}{\varphi(-q^2)^{4k}}\\
&=\sum_{k\geq1} x_{2\alpha+1,  k}  \frac{\varphi(q^2)^{4k-3}}{\varphi(-q^2)^{4k-4}}.
\end{align*} 
If we change $q^2$ to $q$, we obtain
\begin{equation*}
\frac{1}{2}+\sum_{n=1}^\infty v_0(2^{2\alpha+4}n)q^n = \sum_{k\geq1} x_{2\alpha+1,  k}  \frac{\varphi(q)^{4k-3}}{\varphi(-q)^{4k-4}},
\end{equation*}
which is \eqref{T11} with $\alpha+1$ in place of $\alpha$. This completes the proof of \eqref{T11} and \eqref{T12} by induction.

Multiply $q$ on both sides of \eqref{T11} to obtain 
\begin{equation*}
\frac{q}{2}+\sum_{n=1}^\infty v_0(2^{2\alpha+2}n)q^{n+1} = \sum_{j\geq1} x_{2\alpha-1,  j}\ q \frac{\varphi(q)^{4j-3}}{\varphi(-q)^{4j-4}}.
\end{equation*}
Applying the operator $H$ to both sides of the above identity gives
\begin{align*}
\sum_{n=0}^\infty v_0(2^{2\alpha+2}(2n+1))q^{2n+2} &= \sum_{j\geq1} x_{2\alpha-1,  j}H\left(q \frac{\varphi(q)^{4j-3}}{\varphi(-q)^{4j-4}}\right)\\
 &=\sum_{j\geq1} x_{2\alpha-1,  j} \sum_{k\geq1}c_{4j-3,k} \ q^{2k}\frac{\psi(q^8)^{2k-1}}{\varphi(-q^2)^{2k-2}}\\
 & = \sum_{k\geq1} \left(\sum_{j\geq1} x_{2\alpha-1,  j}\ c_{4j-3,k}\right)q^{2k}\frac{\psi(q^8)^{2k-1}}{\varphi(-q^2)^{2k-2}}\\
 & = \sum_{k\geq1} y_{\alpha,k} \ q^{2k}\frac{\psi(q^8)^{2k-1}}{\varphi(-q^2)^{2k-2}}.
\end{align*}
If we now replace $q$ in place of $q^2$, we obtain
\begin{equation*}
\sum_{n=0}^\infty v_0(2^{2\alpha+2}(2n+1))q^n = \sum_{k\geq1} y_{\alpha,k} \ q^{k-1}\frac{\psi(q^4)^{2k-1}}{\varphi(-q)^{2k-2}},
\end{equation*}
which is \eqref{T13}.

We now apply the operator $H$ to both sides of the identity \eqref{T13},
\begin{align*}
\sum_{n=0}^\infty v_0(2^{2\alpha+2}(4n+1))q^{2n} &=\sum_{j\geq1} y_{\alpha,  j} \ H\left(q^{j-1} \frac{\psi(q^4)^{2j-1}}{\varphi(-q)^{2j-2}}\right)\\
&= y_{\alpha, 1}\ \psi(q^4) +\sum_{j\geq1} y_{\alpha,  j+1} \ H\left(q^{j} \frac{\psi(q^4)^{2j+1}}{\varphi(-q)^{2j}}\right)\\
&=y_{\alpha, 1}\ \psi(q^4) +\sum_{j\geq1} y_{\alpha,  j+1} \ \sum_{k\geq1} p_{j,k}\ q^{2k} \frac{\psi(q^4)^{4k+1}}{\varphi(-q^2)^{4k}}\\
&=y_{\alpha, 1}\ \psi(q^4) +\sum_{k\geq1}\left(\sum_{j\geq1} y_{\alpha,  j+1}\ p_{j,k} \right)q^{2k} \frac{\psi(q^4)^{4k+1}}{\varphi(-q^2)^{4k}}\\
&=\sum_{k\geq1} z_{\alpha, k} \ q^{2k-2} \frac{\psi(q^4)^{4k-3}}{\varphi(-q^2)^{4k-4}}.
\end{align*}
On repalcing $q^2$ by $q$, we obtain \eqref{T14}.
\end{proof}
\section{Congruences modulo powers of 2}
For a positive integer $n$, let $\vartheta_2(n)$ be the highest power of 2 that divides $n$, and define $\vartheta_2(0)=\infty$.

\begin{lemma}
For any integer $j, k\geq 1$, we have
\begin{align}
\vartheta_2(m_{j, k}) &\geq 2k-j-1, \label{41}\\
\vartheta_2(n_{j, k}) &\geq 4k-2j-1, \label{44}
\end{align}
and 
\begin{equation}
\vartheta_2(p_{j, k}) \geq 6k-3j-1. \label{45}
\end{equation}
\end{lemma}
\begin{proof}
The proof follows from the properties \eqref{MR1}-\eqref{MR12} and induction.
\end{proof}

\begin{lemma}
For any integer $\alpha\geq1$ and $k\geq1$, we have
\begin{equation}\label{42}
\vartheta_2(x_{2\alpha-1, \, k+1}) \geq 3\alpha+2k-4
\end{equation}
and
\begin{equation}\label{43}
\vartheta_2(x_{2\alpha, \, k+1}) \geq 3\alpha+k.
\end{equation}
\end{lemma}
\begin{proof}
From the values in \eqref{E7} and \eqref{E123}, we note that \eqref{42} and \eqref{43} are true for $\alpha =1$,  and equality holds in each instance for $k=1$.\\
Suppose that \eqref{42} is true for some $\alpha\geq1$.\\
Consider the case $k=1$. In view of \eqref{311}, \eqref{E115} and definition of $M$,
\begin{align*}
x_{2\alpha, 2} &= \sum_{j\geq2} x_{2\alpha-1, j} \ m_{8j-7, 4j-2}\\
& = x_{2\alpha-1, 2} \ m_{9,6} + \sum_{j\geq2} x_{2\alpha-1, j+1} \ m_{8j+1, 4j+2}
\end{align*}
But 
\begin{equation*}
\vartheta_2(x_{2\alpha-1,2} \ m_{9,6})=3\alpha+1
\end{equation*}
and 
\begin{align*}
\vartheta_2(x_{2\alpha-1, j+1} \ m_{8j+1, 4j+2}) &\geq \min_{j\geq2} (\vartheta_2\left(x_{2\alpha-1, j+1})+\vartheta_2(m_{8j+1,4j+2})\right)\\
& \geq 3\alpha+2\geq3\alpha+1
\end{align*}
by hypothesis and \eqref{41}. Thus, $\vartheta_2(x_{2\alpha,2})=3\alpha+1$.\\
Let us now consider the case $k>1$. Then
\begin{align*}
\vartheta_2(x_{2\alpha, \, k+1} ) &= \vartheta_2\left(\sum_{j\geq1} x_{2\alpha-1, j+1} \ m_{8j+1,  4j+k+1} \right)\\
&\geq \min_{j\geq1} (\vartheta_2\left(x_{2\alpha-1, j+1})+\vartheta_2(m_{8j+1,4j+k+1})\right)\\
&\geq3\alpha+2k-2\geq3\alpha+k,
\end{align*}
which is \eqref{43}.

Now suppose \eqref{43} is true for some $\alpha\geq1$.\\
\begin{align*}
\vartheta_2(x_{2\alpha+1, \, k+1} ) &= \vartheta_2\left(\sum_{j\geq1} x_{2\alpha, j+1} \ m_{j,k} \right)\\
&\geq \min_{j\geq1} (\vartheta_2\left(x_{2\alpha, j+1})+\vartheta_2(m_{j,k})\right)\\
&\geq3\alpha+2k-1,
\end{align*}
which is \eqref{42} with $\alpha+1$ in place of $\alpha$. This completes the proof by induction.
\end{proof}
\begin{lemma}
For any integer $\alpha\geq1$ and $k\geq1$, we have
\begin{equation}\label{L14}
\vartheta_2(y_{\alpha, k+1}) \geq 3\alpha+3k+1
\end{equation}
and
\begin{equation}\label{L15}
\vartheta_2(z_{\alpha,k+1}) \geq 3\alpha+6k.
\end{equation}
\end{lemma}
\begin{proof}
In view of \eqref{313}, \eqref{E215} and definition of $N$,
\begin{align*}
y_{\alpha, 2} &= \sum_{j\geq2} x_{2\alpha-1, j} \ n_{8j-7, 4j-2}\\
& = x_{2\alpha-1, 2} \ n_{9,6} + \sum_{j\geq2} x_{2\alpha-1, j+1} \ n_{8j+1, 4j+2}.
\end{align*}
But 
\begin{equation*}
\vartheta_2(x_{2\alpha-1,2} \ n_{9,6})=3\alpha+4
\end{equation*}
and 
\begin{align*}
\vartheta_2(x_{2\alpha-1, j+1} \ n_{8j+1, 4j+2}) &\geq \min_{j\geq2} (\vartheta_2\left(x_{2\alpha-1, j+1})+\vartheta_2(n_{8j+1,4j+2})\right)\\
& \geq 3\alpha+4
\end{align*}
by hypothesis and \eqref{44}. Thus, $\vartheta_2(y_{\alpha,2})=3\alpha+4$.\\
Now suppose $k>1$. Then
\begin{align*}
\vartheta_2(y_{\alpha, \, k+1} ) &= \vartheta_2\left(\sum_{j\geq1} x_{2\alpha-1, j+1} \ n_{8j+1,  4j+k+1} \right)\\
&\geq \min_{j\geq1} (\vartheta_2\left(x_{2\alpha-1, j+1})+\vartheta_2(n_{8j+1,4j+k+1})\right)\\
&\geq3\alpha+4k-1\\
&\geq3\alpha+3k+1,
\end{align*}
which is \eqref{L14}.

From \eqref{314} and definition of $P$,
\begin{align*}
\vartheta_2(z_{\alpha, \, k+1} ) &= \vartheta_2\left(\sum_{j\geq1} y_{\alpha, j+1} \ p_{j,k} \right)\\
&\geq \min_{j\geq1} (\vartheta_2\left(y_{\alpha, j+1})+\vartheta_2(p_{j,k})\right)\\
&\geq3\alpha+6k,
\end{align*}
which is \eqref{L15}. This completes the proof.

\end{proof}
We are now ready to prove Theorems \ref{T2}-\ref{T4} and Corollaries \ref{C1} and \ref{C2}.
\begin{proof}[Proof of \eqref{E1}]
Since $f_i^{16}\equiv f_{2i}^8 \pmod{16}$, we have
\begin{equation*}
\frac{\varphi(q)^{4i}}{\varphi(-q)^{4i}} \equiv 1 \pmod{2^4}.
\end{equation*}
We have \eqref{T11} if $\alpha>0$,
\begin{equation*}
\frac{1}{2}+\sum_{n=1}^\infty v_0(2^{2\alpha+2}n)q^n = \sum_{j\geq1} x_{2\alpha-1,  j}  \frac{\varphi(q)^{4j-3}}{\varphi(-q)^{4j-4}}
\end{equation*}
and \eqref{42}, 
\begin{equation*}
\vartheta_2(x_{2\alpha-1, \, k+1}) \geq 3\alpha-2
\end{equation*}
for $k\geq1$.

Thus,
\begin{align*}
1+\sum_{n=1}^\infty v_0(2^{2\alpha+2}n)q^n &\equiv \frac{1}{2}+\frac{1}{2}\varphi(q) \pmod{2^{3\alpha+2}}\\
&\equiv 1+\sum_{k=1}^\infty q^{k^2}\pmod{2^{3\alpha+2}}.
\end{align*}
This completes the proof of \eqref{E1}.
\end{proof}
\begin{proof}[Proof of \eqref{E3}]
Similarly, we have
\begin{equation*}
\frac{\varphi(q^2)^{2i}}{\varphi(-q)^{2i}} \equiv 1 \pmod{2^2}
\end{equation*}
and \eqref{43},
\begin{equation*}
\vartheta_2(x_{2\alpha, \, k+1}) \geq 3\alpha+1
\end{equation*}
for $k\geq1$.

Thus, by \eqref{T12} if $\alpha>0$,
\begin{equation}\label{P1}
\frac{1}{2}+\sum_{n=1}^\infty v_0(2^{2\alpha+3}n)q^n \equiv  \frac{1}{2} \varphi(q^2) \pmod{2^{3\alpha+3}}
\end{equation}
and by \eqref{E1114},
\begin{equation}\label{P2}
\frac{1}{2}+\sum_{n=1}^\infty v_0(8n)q^n \equiv \frac{1}{2}\varphi(q^2) \pmod{2^3}.
\end{equation}
Congruence \eqref{E3} follows from \eqref{P1} and \eqref{P2}.
\end{proof}
\begin{proof}[Proofs of \eqref{E4} and \eqref{E5}]
We have \eqref{L14}, if $k\geq1$
\begin{equation*}
\vartheta_2(y_{\alpha, \, k+1}) \geq 3\alpha+3k+1
\end{equation*}
and
\begin{equation*}
\frac{1}{\varphi(-q)^{2i}} \equiv 1 \pmod{2^2}.
\end{equation*}
Thus, by \eqref{T13} if $\alpha>0$,
\begin{equation*}
\sum_{n=0}^\infty v_0(2^{2\alpha+2}(2n+1))q^n \equiv y_{\alpha,1} \ \psi(q^4) + y_{\alpha,2} \ q\psi(q^4)^3 \pmod{2^{3\alpha+6}},
\end{equation*}
which implies that
\begin{equation}\label{315}
\sum_{n=0}^\infty v_0(2^{2\alpha+2}(4n+3))q^n \equiv  y_{\alpha,2} \ \psi(q^2)^3 \pmod{2^{3\alpha+6}}.
\end{equation}
Congruences \eqref{E4} and \eqref{E5} follow from \eqref{S13} and \eqref{315}.
\end{proof}
\begin{proof}[Proofs of \eqref{E55} -- \eqref{E9}]
We have \eqref{L15}, if $k\geq1$
\begin{equation*}
\vartheta_2(z_{\alpha, \, k+1}) \geq 3\alpha+6k
\end{equation*}
and
\begin{equation*}
\frac{1}{\varphi(-q)^{4i}} \equiv 1 \pmod{2^3}.
\end{equation*}
Thus, by \eqref{T00} and \eqref{T14} if $\alpha>0$,
\begin{equation}\label{318}
\sum_{n=0}^\infty v_0(16n+4)q^n \equiv 3  \psi(q^2) + 64 q\psi(q^2)^5 \pmod{2^{9}}
\end{equation}
and
\begin{equation}\label{316}
\sum_{n=0}^\infty v_0(2^{2\alpha+2}(4n+1))q^n \equiv z_{\alpha,1} \ \psi(q^2) + z_{\alpha,2} \ q\psi(q^2)^5 \pmod{2^{3\alpha+9}}.
\end{equation}
Congruence \eqref{E55} follows from \eqref{318} and \eqref{316}. Extracting terms involving even powers of $q$, we obtain
\begin{equation}\label{319}
\sum_{n=0}^\infty v_0(32n+4)q^n \equiv 3 \psi(q) \pmod{2^{9}}.
\end{equation}
and
\begin{equation}\label{317}
\sum_{n=0}^\infty v_0(2^{2\alpha+2}(8n+1))q^n \equiv z_{\alpha,1} \ \psi(q) \pmod{2^{3\alpha+9}}.
\end{equation}
Congruence \eqref{E6} is immediate from \eqref{EE11}, \eqref{EE2}, \eqref{319} and \eqref{317}. Congruences \eqref{E8} and \eqref{E9} follow from \eqref{EE3}, \eqref{EE4}, \eqref{319} and \eqref{317}.
\end{proof}
\section{Congruences involving other moduli}
In this section, we establish some congruences modulo 13, 25 and 27. In order to do so, certain identities involving the generating function for $v_0(n)$ are obtained.
\begin{theorem}
We have
\begin{equation}\label{T5}
\sum_{n=0}^\infty v_0(32n+4)q^n = 3\psi(q)+512q\frac{\psi(q)^9}{\varphi(-q)^8}.
\end{equation}
\end{theorem}
\begin{proof}
In view of \eqref{T00}, \eqref{L11} and \eqref{L13},
\begin{align*}
\sum_{n=0}^\infty v_0(16n+4)q^n &= 3\psi(q^2)+64q\frac{\psi(q^2)^5}{\varphi(-q^2)^8}\ \varphi(q)^4 \\
&=3\psi(q^2)+64q\frac{\psi(q^2)^5}{\varphi(-q^2)^8} \ (\varphi(q^2)^4+16q^2\psi(q^4)^4+8q\psi(q^2)^4).
\end{align*}
Which yields \eqref{T5}
\end{proof}
\begin{theorem}\label{T6}
For any integer $n\geq0,$
\begin{equation}\label{S23}
v_0(416n+r)\equiv 0\pmod{13}, \quad r\in\{132, 164, 228, 292, 356, 388\}. 
\end{equation}
\end{theorem}
\begin{proof}
By the binomial theorem, it is easy to verify that 
\begin{equation*}
f_{k}^{13}\equiv f_{13k} \pmod{13}
\end{equation*}
for any positive integer $k$. Thus
\begin{equation}\label{E221}
\frac{\psi(q)^9}{\varphi(-q)^8} = \frac{f_2^{26}}{f_1^{25}} \equiv f_1 \frac{f_{26}^2}{f_{13}^2} \pmod{13}.
\end{equation}
Using \eqref{E221}, it follows that
\begin{align*}
\sum_{n=0}^\infty v_0(32n+4)q^n &\equiv 3\psi(q)+512qf_1 \frac{f_{26}^2}{f_{13}^2} \pmod{13}\\
&\equiv 3 \sum_{n=0}^\infty q^{n(n+1)/2}+5q\sum_{n=-\infty}^\infty (-1)^n q^{n(3n-1)/2} \frac{f_{26}^2}{f_{13}^2} \pmod{13}.
\end{align*}
Since $n(n+1)/2\not\equiv4, 5, 7, 9, 11, 12 \pmod{13}$ for any positive integer $n$ and $n(3n-1)/2 \not\equiv 3, 4, 6, 8, 10, 11 \pmod{13}$ for any integer $n$. The coefficients of $q^{13n+r}$ for $r\in \{4, 5, 7, 9, 11, 12\}$ in $\sum_{n=0}^\infty v_0(32n+4)q^n$ are congruent to $0 \pmod{13}$. This completes the proof.
\end{proof}

\begin{theorem}\label{T7}
For any integer $n\geq0,$
\begin{equation}\label{S24}
v_0(160n+r)\equiv 0\pmod{25}, \quad r\in\{68, 132\}. 
\end{equation}
\end{theorem}
\begin{proof}
By the binomial theorem, it is immediate that
\begin{equation}\label{E222}
\frac{\psi(q)^9}{\varphi(-q)^8} = \frac{f_2^{26}}{f_1^{25}} \equiv f_2 \frac{f_{10}^5}{f_5^5} \pmod{25}.
\end{equation}
Using \eqref{E222}, it follows that
\begin{align*}
\sum_{n=0}^\infty v_0(32n+4)q^n &\equiv 3\psi(q)+512qf_2 \frac{f_{10}^5}{f_5^5} \pmod{25}\\
&\equiv 3 \sum_{n=0}^\infty q^{n(n+1)/2}+12q\sum_{n=-\infty}^\infty (-1)^n q^{n(3n-1)} \frac{f_{10}^5}{f_5^5} \pmod{25}.
\end{align*}
Since $n(n+1)/2\not\equiv 2, 4 \pmod{5}$ for any positive integer $n$ and $n(3n-1) \not\equiv 1, 3 \pmod{5}$ for an integer $n$, we have that the coefficients of $q^{5n+2}$ and $q^{5n+4}$ in $\sum_{n=0}^\infty v_0(32n+4)q^n$ are congruent to $0 \pmod{25}$.
\end{proof}
\begin{theorem}\label{T10}
For any nonnegative integer $n$, we have
\begin{align}
v_0(3^2(96n+68)) &\equiv  0\pmod{27}, \label{E22}\\
v_0(3^3(96n+44)) &\equiv 0 \pmod{27}, \label{E23}\\
v_0(3^3(96n+76)) &\equiv 0 \pmod{27}. \label{E24}
\end{align}
	
\end{theorem}
\begin{theorem}\label{T8}
	Let $p\geq3$ be any prime. Then for any nonnegative integer $n$ and $\alpha$, we have
	\begin{equation}\label{E25}
	v_0(p^{2\alpha}\ 9^{2}(32n+4))  \equiv 0 \pmod{27}, \quad \text{if} \,\, n \, \text{is not represented as a triangular number},
	\end{equation}
	whence
	\begin{equation}\label{E26}
	v_0(p^{2\alpha+1}\ 9^{2}(32pn+32i+4p))  \equiv 0 \pmod{27},\,\ \text{where}\,\ i\in \{1, 2, 3, \ldots, p-1\}
	\end{equation}
	and
	\begin{equation}\label{E27}
	v_0(p^{2\alpha}\ 9^{2}(32pn+32j+4))  \equiv 0 \pmod{27},
	\end{equation}
	whenever $8j+1$ is quadratic nonresidue modulo $p$.
\end{theorem}
\begin{theorem}\label{T9}
Let $p\geq3$ and $p_1 \geq 5$ be any prime. Then for any nonnegative integer $n$ and $\alpha$, we have
\begin{equation}\label{E28}
v_0(p^{2\alpha} (96n+4))  \equiv 0 \pmod{27}, \quad \text{if} \,\, n\neq k(3k+1)/2,\,\ k\in \mathbb{Z},
\end{equation}
whence
\begin{equation}\label{E29}
v_0(p^{2\alpha} p_1(96p_1n+96i+4p_1))  \equiv 0 \pmod{27},\,\ \text{where}\,\ i\in \{1, 2, 3, \ldots, p_1-1\}
\end{equation}
and
\begin{equation}\label{E30}
v_0(p^{2\alpha} (96p_1n+96j+4))  \equiv 0 \pmod{27},
\end{equation}
 whenever $24j+1$ is quadratic nonresidue modulo $p_1$.
\end{theorem}
\begin{proof}
From \cite[Entry 31, p. 49]{BCB3}, we have
\begin{align}
\psi(q)&=f(q^3, q^6)+q\psi(q^9), \label{E31}\\
\varphi(-q)&=\varphi(-q^9)-2qf(-q^3, -q^{15}). \label{E32}
\end{align}
Employing \eqref{E31} and \eqref{E32} in \eqref{T5}, then extracting the terms involving  $q^{3n}$ and $q^{3n+1}$ from both sides,
\begin{equation}\label{E33}
\sum_{n=0}^\infty v_0(96n+4) q^n \equiv 3 f(q, q^2) \pmod{27}
\end{equation}
and
\begin{equation}\label{E34}
\sum_{n=0}^\infty v_0(96n+36) q^n \equiv 3\psi(q^3)- \varphi(-q^3) \frac{f_2^9}{f_1^9} \pmod{27}.
\end{equation}
From \cite[Eq. 3.1]{PCT}, we have
\begin{equation}\label{E35}
\frac{f_2^3}{f_1^3}=\frac{f_6}{f_3}+3q\frac{f_6^4f_9^5}{f_3^8f_{18}}+6q^2 \frac{f_6^3f_9^2f_{18}^2}{f_3^7}+12q^3\frac{f_6^2f_{18}^5}{f_3^6f_9}.
\end{equation}
Invoking \eqref{E35} in \eqref{E34}, then extracting the terms involving $q^{3n}$ from both sides,
\begin{align}
\nonumber\sum_{n=0}^\infty v_0(96(3n)+36) q^n &\equiv 2\frac{f_2^2}{f_1}-9q\frac{f_2^3f_6^5}{f_1^6f_3} \\
&\equiv 2 \psi(q)-9q\psi(q^9) \pmod{27}.\label{E36}
\end{align}
It follows from \eqref{E31} and \eqref{E36} that
\begin{align} 
\sum_{n=0}^\infty v_0(9(96n+4)) &\equiv 2 f(q, q^2) \pmod{27}, \label{E37}\\
\sum_{n=0}^\infty v_0(9(96n+36))q^n&\equiv -7 \psi(q^3) \pmod{27}\label{E38}
\end{align}
 and 
\begin{equation}\label{E39}
v_0(9(96n+68)) \equiv 0\pmod{27}
\end{equation}
for all nonnegative integers $n$. Congruences \eqref{E22}-\eqref{E24} follow from \eqref{E38} and \eqref{E39}.

It follows from \eqref{E38} that
\begin{equation*}
\sum_{n=0}^\infty v_0(9^2(32n+4))q^n\equiv -7 \psi(q) \pmod{27}.\label{E40}
\end{equation*}
In view of \eqref{EE11} and \eqref{EE2} and the above equation, it is easy that for any prime $p\geq3$
\begin{align}
\sum_{n=0}^\infty v_0(p^{2\alpha}\ 9^{2}(32n+4))  &\equiv -7\psi(q) \pmod{27},\label{E41}\\
\nonumber	& -7 \sum_{n=0}^\infty q^{n(n+1)/2} \pmod{27},
\end{align}
from which we obtain \eqref{E25}. Congruences \eqref{E26} and \eqref{E27} follow from \eqref{EE3}, \eqref{EE4} and \eqref{E41}.

Using \eqref{E31} in \eqref{E41}, then extracting the terms involving $q^{3n}$, we have
\begin{equation}
\sum_{n=0}^\infty v_0(p^{2\alpha}\ 9^{2}(96n+4))  \equiv -7 f(q, q^2) \pmod{27},\label{E42}
\end{equation}
for any integer $\alpha\geq 0$ and prime $p\geq3$.

From the definition,
\begin{equation}\label{E46}
f(q, q^2) = \sum_{n=-\infty}^\infty q^{n(3n+1)/2}.
\end{equation}
For any prime $p_1\geq 5$, the $p-$ dissection identity for $f(q, q^2)$ is given by (see \cite{CG}),
\begin{equation}\label{E43}
f(q, q^2) = \mathop{\sum_{k=-\frac{p_1-1}{2}}}_{k\neq \frac{\pm p_1-1}{6}}^{\frac{p_1-1}{2}} q^{\frac{3k^2+k}{2}} f(q^{\frac{3p_1^2+(6k+1)p_1}{2}}, q^{\frac{3p_1^2-(6k+1)p_1}{2}}) + q^{\frac{p_1^2-1}{24}}f(q^{p_1^2}, q^{2p_1^2}).
\end{equation}
Further, if $-(p_1-1)/2\leq k \leq (p_1-1)/2$ and $k \neq (\pm p_1-1)/6$,
\begin{equation*}
\frac{3k^2+k}{2} \not \equiv \frac{p_1^2-1}{24} \pmod{p_1}.
\end{equation*}
It is easy from \eqref{E33}, \eqref{E43} and mathematical induction that for any prime $p_1\geq5$ and $\alpha\geq 0$,
\begin{equation}\label{E44}
\sum_{n=0}^\infty v_0 (p_1^{2\alpha}(96n+4)) q^n \equiv 3 f(q, q^2) \pmod{27}.
\end{equation}
Thus, from \eqref{E37}, \eqref{E42} and \eqref{E44},
\begin{equation}\label{E45}
\sum_{n=0}^\infty v_0 (p^{2\alpha}(96n+4)) q^n \equiv K f(q, q^2) \pmod{27},
\end{equation}
where $K=2$ if $p=3$ and $\alpha=1$, $K=-7$ if $p=3$ and $\alpha>1$,
and $K=3$ if $p\geq5$.

\noindent Congruence \eqref{E28}  follows from \eqref{E45} and \eqref{E46}.

Now, utilizing \eqref{E43} in \eqref{E45} and extracting the terms involving $q^{p_1n+\frac{p_1^2-1}{24}}$, we obtain
\begin{equation}\label{E47}
\sum_{n=0}^\infty v_0 (p^{2\alpha}p_1(96n+4p_1)) q^n \equiv K f(q^{p_1}, q^{2p_1}) \pmod{27}
\end{equation}
from which, we obtain congruence \eqref{E29}.

Finally, according to \eqref{E43} and \eqref{E45}, if 
\begin{equation*}
j \not \equiv \frac{3k^2+k}{2}  \pmod{p_1}
\end{equation*}
that is, $(24j+1)$ is a quadratic non-residue modulo $p$, for prime $p_1\geq5$ and $|k| \leq (p_1-1)/2$, then
\begin{equation*}
\sum_{n=0}^\infty v_0 (p^{2\alpha}(96(p_1n+j)+4)) q^n \equiv 0 \pmod{27}
\end{equation*}
which is \eqref{E30}.
\end{proof}
\section{Closing remarks}
We end this paper by remarking that many congruences modulo certain number of the form $2^\alpha 3^\beta 5^\gamma 13$ can be presented. For instance, 
from \eqref{S23}, \eqref{S24}, \eqref{E22} and \eqref{E9}, we have
\begin{equation*}
v_0(56160n+39492) \equiv 0 \pmod{2^9 3^3 5^2 13}.
\end{equation*}

\end{document}